\documentclass[12pt,reqno]{amsart}
\newtheorem{innercustomgeneric}{\customgenericname}
\providecommand{\customgenericname}{}
\newcommand{\newcustomtheorem}[2]{%
  \newenvironment{#1}[1]
  {%
   \renewcommand\customgenericname{#2}%
   \renewcommand\theinnercustomgeneric{##1}%
   \innercustomgeneric
  }
  {\endinnercustomgeneric}
}

\newcustomtheorem{customthm}{Theorem}
\newcustomtheorem{customlemma}{Lemma}
\newcustomtheorem{customcoro}{Corollary}
\newcustomtheorem{customprop}{Proposition}
\newcustomtheorem{customfact}{Fact}
\usepackage{tikz}
\usetikzlibrary{matrix}
\usepackage{lineno}
\usepackage{amssymb}
\usepackage{verbatim}
\usepackage{hyperref}
\usepackage{array}
\usepackage{latexsym}
\usepackage{enumerate}
\usepackage{amsmath}
\usepackage{amsfonts}
\usepackage{amsthm}
\usepackage{color}
\usepackage[english]{babel}

\usepackage{mathrsfs}
\usepackage{hyperref}

\newtheorem{theorem}{Theorem}[section]

\newtheorem{rema}[theorem]{Remark}

{\theoremstyle{definition}
}
\newtheorem*{theo*}{Theorem}

\newtheorem*{rem*}{Remark}

\newtheorem*{corollary*}{Corollary}

\def\cX{X}

\newcommand{\s}{\vspace{0.2cm}}

\newcommand{\we}{Weierstrass }

\begin{document}
\title[Riemann surfaces with transitive action on Weierstrass points]{Infinitely many Riemann surfaces with a\\ transitive action on the Weierstrass points}

\author{Sebasti\'an Reyes-Carocca}
\address{Departamento de Matem\'aticas, Facultad de Ciencias, Universidad de Chile, Las Palmeras 3425, \~{N}u\~{n}oa, Santiago, Chile.}
\email{sebastianreyes.c@uchile.cl}
\author{Pietro Speziali}
\address{Instituto de Matem\'atica, Estat\'istica e Computa\c{c}\~ao Cient\'ifica, Universidade Estadual de Campinas, Campinas, SP 13083-859, Brazil.}
\email{speziali@unicamp.br}

\thanks{Partially supported by Fondecyt Regular grants 1220099 and 1190991}
\keywords{Riemann surfaces, \we points, automorphisms}

\begin{abstract} In this short note, we prove the existence of infinitely many pairwise non-isomorphic, non-hyperelliptic Riemann surfaces with automorphism group acting transitively on the \we points. We also find all compact Riemann surfaces  with automorphism group acting transitively on the \we points, under the assumption that they are simple.

\end{abstract}
\maketitle
\thispagestyle{empty}
\section{Introduction and statement of the result} 

Weierstrass point theory is one of the main tools in the study of the geometry of compact Riemann surfaces. Since its inception in Hurwitz's foundational paper \cite{Hu}, this theory has tremendously evolved due to tools from different areas, such as projective geometry, complex analysis and numerical semigroup theory.

\s

In this note, we address the general problem of deciding whether the automorphism group of a given Riemann surface acts transitively on the set of its \we points. 

\s

Although this problem is not new, the literature shows very few examples of Riemann surfaces with this property. This fact accounts for the difficulty of the issue of locating the \we points, even in low genera.  Interestingly, all the known examples  are very symmetric and have played an important role in the theory of Riemann surfaces for several reasons. 

\s

Here we mention some remarkable examples.

 \begin{enumerate}
 \item The Accola-Maclachlan curve of genus $g \geqslant 2$ (hyperelliptic) has automorphism group acting transitively on the \we points. 
 
 \s

\item Fermat's quartic is the unique curve among the Fermat's curves with automorphism group acting transitively on the Weierstrass points.
 
 \s
 
 \item Klein's quartic, Macbeath's curve and Bring's curve have automorphism group acting transitively on the \we points. The first two are Hurwitz curves, whereas the latter is the most symmetric among genus four curves.
 
 \s
 
 \item There is no Hurwitz curve of genus $g>14$ with automorphism group acting transitively on the \we points.
 
\end{enumerate}We refer to \cite{KMV},  \cite{laings} and  \cite{mv} for these and other examples. See also \cite{SW}.

\s

Up to the authors' knowledge,  the Accola-Maclachlan curves are the only examples of  infinitely many  Riemann surfaces with transitive action on the \we points. This  was pointed out by Laing and Singerman in \cite{laings}. Concretely, in the conclusion of that paper, the authors mentioned that ``it would be interesting to search for other non-hyperelliptic Riemann surfaces whose automorphism group acts transitively on the Weierstrass points."

\s

This short note is devoted to proving that the non-hyperelliptic Riemann surfaces of genus three with automorphism group acting transitively on the \we points form a complex one-dimensional family.

\s

As a consequence, we deduce our main theorem.

\begin{theo*} 
There are infinitely many pairwise non-isomorphic, non-hyperelliptic compact Riemann surfaces with automorphism group acting transitively on the \we points.
\end{theo*}

The theorem above is a rather surprising result, as it contradicts the expectation of the non-existence of infinitely many non-hyperelliptic Riemann surfaces with a transitive action on the Weierstrass points. Such an expectation is clearly expressed in \cite{mv} and \cite{SC}. More precisely, in the introduction of \cite{mv}, Magaard and V\"{o}lklein state that ``among hyperelliptic curves there are infinitely many examples of curves with automorphism group acting transitively on the Weierstrass points. We expect no other
infinite family of curves with this property to exist". Later
Singerman in \cite{SC} states the aforementioned expectation as a conjecture (indeed, Singerman makes three conjectures  that this paper falsifies; see \cite[Section 5]{SC}).

\s

As another consequence of our result, we find all the compact Riemann surfaces of genus $g >2$ with automorphism group acting transitively on the \we points and having all such points weight 1.

\s

At the end of this paper, we shall discuss some future directions and open questions.

\section{A family of Riemann surfaces} We recall that the moduli space   $\mathscr{M}_3$ of compact Riemann surfaces of genus $g=3$ is a complex analytic space of dimension six and that its singular locus agrees with the points representing surfaces admitting non-trivial automorphisms, with the exception of the hyperelliptic surfaces having automorphism group of order two; see \cite{Oort13}. Let  $$\mathscr{K}_3 \subset \mbox{Sing}(\mathscr{M}_3)$$be the locus formed by points represented by the smooth  projective algebraic curves $$\mathscr{C}_t : x^4 +y^4 +z^4+ t(x^2y^2 +y^2z^2 +z^2x^2) = 0$$ where $t$ is a complex number different from $-1$ and  $\pm 2.$   The members of $\mathscr{K}_3$ are non-hyperelliptic and have a group of automorphisms isomorphic to the symmetric group $\mathbf{S}_4$ of order 24, generated by the transformations $$[x:y:z]\mapsto [x:z:-y] \,\, \mbox{ and } \,\, [x:y:z]\mapsto [z:y:x].$$

This group agrees with the (full) automorphism group of  each member of $\mathscr{K}_3$ with only two exceptions:   Fermat's quartic and Klein's quartic $$x^4 + y^4 + z^4 =0 \,\, \mbox{ and } \,\,x^3y+y^3z+z^3x=0,$$which are obtained for $t=0$ and $t=t_0:=\tfrac{3}{2}(-1\pm 7i)$ respectively.  Kuribayashi and  Sekita  proved  in \cite{ks} that if $t$ and $t'$ are different from $t_0$  then  $$\mathscr{C}_{t} \cong \mathscr{C}_{t'} \,\, \iff \,\, t=t'.$$These curves were considered by Kuribayashi and Komiya in \cite{KK79} and  are known as {\it Kuribayashi-Komiya quartics}. We refer to  \cite{KFT} for further properties of this family, also known as  {\it KFT pencil}. See also \cite{bars}, \cite{dolga} and \cite{ik}.

\section{Proof of the main theorem}

 The proof of the main theorem will be a consequence of the following four facts.

\s

\begin{customfact}{1} Let $\cX$ be a non-hyperelliptic compact Riemann surface of genus $g=3$. If the  automorphism group of $\cX$ acts transitively on the Weierstrass points then $\cX \in \mathscr{K}_3.$
\end{customfact}

\begin{proof}
Assume that the automorphism group $G$ of $X$ acts transitively on the Weierstrass points. By \cite[Theorem 1]{kato}, the weight of each of these points is either equal to 1 or 2. Since the sum of the weights of all the points of $X$ is equal to 24, the Riemann surface $X$ has either 12 or 24 \we points. In particular, the order of $G$ must satisfy $$|G|=12|G_P| \,\, \mbox{ or } \,\, |G|=24|G_P|$$ where $G_P$ stands for the stabiliser subgroup of a \we point $P$ of $X.$ It follows that the order of $G$ must be divisible by 12 and consequently, by \cite[Theorem 16]{bars}, the Riemann surface either belongs to $\mathscr{K}_3$ or is isomorphic to the Picard curve $$\mathscr{Q} : x^4 +y^4 +z^3x= 0.$$

This latter curve is the unique non-hyperelliptic Riemann surface of genus $g=3$ with exactly 48 automorphisms (see, for instance, \cite{CP} and \cite{Shaska}).

\s
 To prove the fact it suffices to verify that $X$ is not isomorphic to $\mathscr{Q}.$ Let  $$Q = [0:0:1] \in \mathscr{Q}$$ and denote by $\mathbb{C}(\bf{x},\bf{y})$  the function field of $\mathscr{Q}$. A straightforward  computation shows that $1/{\bf x}$ has a pole of order $4$ at $Q$ while it is regular elsewhere, and that ${\bf y}/{\bf x}$ has a pole of order $3$ at $Q$ while it is regular elsewhere. In other words,  $Q$ is a \we point of $\mathscr{Q}$, its gap sequence is $\{1,2,5\}$ and its weight is 2.  Now, the transitivity implies that all the \we points of $\mathscr{Q}$ have weight equal to $2$, which is the maximal possible weight in genus three. However, following \cite[Theorem 1.2]{keem}, there are exactly two Riemann surfaces of genus three on which all \we points are of maximal weight; these surfaces being $\mathscr{C}_0$ and $\mathscr{C}_3.$ Clearly, none of them is isomorphic to $\mathscr{Q}$ and the proof of the fact is done. 
\end{proof}
 
Although the following fact is known (see, for instance, \cite{FS}), we include a proof for the sake of completeness. 
 
  \begin{customfact}{2} \label{f2}
  Assume that $X$ is isomorphic to $\mathscr{C}_t$ with $t \neq 0, -1, \pm 2, t_0$ 
and let $G$ denote its automorphism group. Then the signature of the action of  $G$ on $X$ is $(0;2,2,2,3)$. In other words, the quotient $X/G$ is rational, the associated regular covering map $$X \to X/G \cong \mathbb{P}^1$$ramifies over exactly four values, and the corresponding fibers consist of ramification points with stabilisers of orders $2,2,2$ and $3$ respectively.
 \end{customfact}

\begin{proof} 
We recall that the signature of the action of a group isomorphic to $\mathbf{S}_4$ on a compact Riemann surface of genus $g=3$ is either $(0;2,2,2,3)$ or $(0; 3,4,4)$. 
Consider the automorphism of $\mathbb{P}^2$ given by
$$
\Phi([x:y:z]) = [-y:x:z]
$$and observe that $\Phi$ has only two fixed points$$P_1=[-i:1:0] \,\,\, \mbox{ and } \,\,\, P_2=[1:-i:0].$$
Note that $\Phi$ restricts to an automorphism of order 4 of $\mathscr{C}_t$ for each $t$. A straightforward computation shows that $P_1, P_2$ belong to $\mathscr{C}_t$ if and only if $t=2$, and therefore $$G \geqslant \langle \Phi \rangle \cong C_4$$is a group of automorphisms of $X$ acting without fixed points. This, coupled with the fact that ${\bf S}_4$ possesses a single conjugacy class of cyclic subgroups of order $4$, shows that there is no point of $X$ with $G$-stabiliser of order divisible by $4$. Hence, the signature cannot be $(0;3,4,4)$ and the claim follows.
\end{proof}

\begin{customfact}{3} \label{f3}
If $X$ is isomorphic to $\mathscr{C}_t$ with $t \neq 0, -1, \pm 2, t_0$ 
and there is a \we point of $X$ fixed by an involution of $X$ then the action of the automorphism group $G$ of $X$ is transitive on the \we points. \end{customfact}

\begin{proof}
 We recall that the group $\mathbf{S}_4$ has two conjugacy classes of cyclic subgroups of order two (a standard reference for permutation groups is \cite{Cameron}; a nice and concise list of properties of $\mathbf{S}_4$, including presentations, can be also found at \cite{database}). In terms of the presentation$$G \cong \mathbf{S}_4= \langle a,b,c,d : a^2, b^2,(ab)^2, c^3, cac^{-1}ab, cbc^{-1}a, d^2, dadab, [d,b], (dc)^2 \rangle,$$ these conjugacy classes are represented by $\langle d \rangle$ (of length 6) and by $\langle b \rangle$  (of length 3).  A computation shows that the action of $G$ on $X$ is topologically equivalent to the one represented by the generating vector $$(b,db,dc,c^{-1})$$(see \cite{Brou} for a precise definition of generating vector) and consequently, each involution $\iota$ of $G$ acts with fixed points. More precisely,  if $\iota$ is conjugate to $d$ then $\iota$ has 4  fixed points, and such points belong to two distinct orbits. By contrast,  if $\iota$ is conjugate to $b$ then $\iota$ has 4  fixed points, and such points belong to the same orbit (see Macbeath's formula \cite{mcB}).

\s

Now, assume that there exists a \we point $P$ fixed by an involution. The weight of points fixed by involutions was investigated by Towse. In fact, following his main results in \cite{towse}, the expected weight of a point fixed by an involution fixing $4$ points is equal to $0$, that is, these points are not \we points unless the image points in the quotient curve are \we points, either for the canonical or another linear series. In addition,  the parity of the actual weight  of such a point must agree with the parity of the expected weight. It follows that  the weight of $P$ is even; hence it is equal to $2$. If $P$ is fixed by an involution conjugate to $d$ then $X$ has  24 \we points of weight $2$; a contradiction. Thus, one sees that $P$ must be fixed by an involution conjugate to $b$. Thereby, the orbit of $P$ consists of 12 \we points of weight $2$, and the action is transitive, as claimed.
\end{proof}

 \begin{customfact}{4} \label{f4}
If $X \in \mathscr{K}_3$  then the automorphism group $G$ of $X$ acts transitively on the Weierstrass points.
\end{customfact}

\begin{proof} Since Klein's quartic and Fermat's quartic are known to have a transitive action on the \we points \cite{laings}, we only need to consider the remaining cases, namely, when $X$ is isomorphic to $\mathscr{C}_t$ where $t \neq 0, -1, \pm2, t_0.$ As proved in Fact \ref{f2}, in such a case the $G$-stabiliser subgroup of each point of $X$ is trivial, of order two or of order three.
\s

 Assume that there exists a \we point $P$ with trivial stabiliser. Then each point in the $G$-orbit of $P$ is a \we point. In particular, we have $24$ \we points of weight $1$, and the action is transitive. 
 
 \s
 
Now, suppose  that there exists a \we point $P$ fixed by an automorphism of order three. In such a case, $X$ has 8 \we points. It follows from the fact that $X$ is non-hyperelliptic that there exists a \we point $R$ which does not belong to the $G$-orbit of $P$. If the stabiliser of $R$ were trivial, then $X$ would have 32 \we points; a contradiction. It then follows from the preceding discussion that $R$ must be fixed by an involution conjugate to $b$ and therefore  the surface possesses, in addition to $P$ and its whole $G$-orbit, 12 \we points of weight 2. This is impossible and therefore there is no \we point among the fixed points of automorphisms of order three. 

\s

Finally, in the remaining case (namely, if there is a \we point fixed by an involution) the transitivity has been already proved. The proof of the fact then follows.
 \end{proof}

The previous four facts can be summarised in the following theorem.

\begin{customthm}{1}
 A non-hyperelliptic compact Riemann surface $X$ of genus three has automorphism group acting transitively on the \we points if and only if $X$ belongs to the family $\mathscr{K}_3.$
\end{customthm}

The proof of the main theorem follows directly from the result above.

\section{Simple \we points}

Laing and Singerman in \cite{laings} searched for all Riemann surfaces of genus $g >2$ admitting a transitive action on the \we points, but under the additional assumption that such  points are simple (namely, of weight 1). They succeeded in giving a partial answer to this problem by taking advantage of the geometric richness in the case of platonic Riemann surfaces, leaving the case of $\mathbf{S}_4$ acting on genus $g=3$ with signature $(0; 2,2,2,3)$ as an open problem. In fact, referring to that situation, they pointed out that  ``it is unclear whether there could be any [Riemann surface in the family $\mathscr{K}_3$] where the group does not fix any Weierstrass points or whether the weight of the Weierstrass points is equal to one". 

\s

Observe that the proof of Fact 4 shows that there are only two scenarios for the distribution of the \we points of $X \in \mathscr{K}_3,$ namely:
\begin{enumerate}
\item the \we points form a long orbit and they are simple, or
\item  the \we points form a short orbit of length 12 and they are double.
\end{enumerate}

Following the results of \cite{keem}, the latter case occurs if and only if $X$ is Fermat's quartic or $X$ is isomorphic to $$\mathscr{C}_3 : x^4+y^4+z^4+3(x^2y^2+y^2z^2+z^2x^2)=0.$$

\s
All the above coupled with \cite[Theorem 15]{laings} is the proof of the following result.

\begin{customthm}{2}
There is a transitive action on the \we points on a compact Riemann surface $\cX$ of genus $g > 2$ and the \we points are simple if and only if one of the following statements holds.
\begin{enumerate} 
\item $g=4$ and $\cX$ is isomorphic to Bring's curve. 
\item $g=3$ and $\cX \in \mathscr{K}_3$  is isomorphic $\mathscr{C}_t$ for some $t$ different from  $0$ and $3.$
\end{enumerate}
\end{customthm}

The theorem above answers a question posed by Laing and Singerman and completes \cite[Theorem 15]{laings}. 

\s

For the sake of completeness, we recall that the \we points of Bring's curve were found by Edge in \cite{edge}; see also \cite[Section 3.1]{bra}.

\begin{rema}
The \we points of  Fermat's quartic were obtained by Hasse in \cite{hasse}. Besides, the \we points of $\mathscr{C}_3$ can be easily computed, since they are the fixed points of the involutions conjugate to $[x:y:z] \mapsto [-x:y:z]$. Explicitly, the \we points of $\mathscr{C}_3$ are   $$[0:1:\zeta], [1: \zeta: 0], [\zeta:0:1] \,\, \mbox{ where }\,\, \zeta = \pm \sqrt{\tfrac{-3\pm\sqrt{5}}{2}}.$$As simple Weierstrass points  are precisely ordinary flexes,
 the \we points of $\mathscr{C}_t$ for  $t \neq 0, -1, \pm 2, 3, t_0$ are the points  of $\mathscr{C}_t$ at which the Hessian determinant vanishes.
\end{rema}

\section{Future directions and final comments}

We end this note with some remarks.

\subsection*{The action on $\mathscr{C}_3$} It is worth recalling that the family $\mathscr{K}_3$ is equisymmetric, namely, it is a closed irreducible algebraic subvariety of $\mathscr{M}_3$ and the action of $\mathbf{S}_4$ on its members is the same, in the sense of topological actions (see \cite{b22} for a precise definition of topological action). Then it naturally arises the problem of understanding the dichotomy between $\mathscr{C}_3$ (with double \we points) and $\mathscr{C}_t$ with $t \neq 0,3$ (with simple \we points), and if such a dichotomy can be read off from the geometry of the action.

\subsection*{The Jacobian variety} Let $j: X \hookrightarrow JX$ denote the Abel-Jacobi map of $X$ and let $\mathscr{W}_X$ be the group generated by the images under $j$ of Weierstrass points of $X$. If $X$ is hyperelliptic then it is well-known that $\mathscr{W}_X$ agrees with the $2$-torsion points of $JX$. By contrast, for non-hyperelliptic surfaces of genus 3 the structure of $\mathscr{W}_X$ have been determined only for some special cases: for instance, Klein's quartic in \cite{ref6}, Fermat's quartic in \cite{ref7} and the Picard curve  
in \cite{ref2}. It would be interesting to understand the extent to which the transitivity on the \we points may help in the problem of determining the algebraic structure of $\mathscr{W}_X$.  See also  \cite{girard} and \cite{kamel2}.

\subsection*{New examples} Since all the known examples of non-hyperelliptic Riemann surfaces with transitive action on the \we points turn out to have  genus at most seven, it would be interesting to search for new examples (and infinite collection of examples) in higher genera, if they exist. If such Riemann surfaces do not exist for arbitrary genus, it would be interesting to find an upper bound for their genera.

\subsection*{The case of positive characteristic} Transitive actions on \we points can be also studied for (projective, non-singular, absolutely irreducible) algebraic curves  defined over algebraically closed fields of positive characteristic. However, in this context the relationship between \we point and automorphisms becomes much more involved than in the complex case. In fact, the following two problems may arise:
\begin{enumerate}
\item the curve might be non-classical, that is, the gap sequence at a generic point can   be different from $\{1,\ldots,g\}$, where $g$ is the genus of the curve, and
\item the order of the automorphism group might exceed the Hurwitz bound $84(g-1)$. 

\end{enumerate}

For instance, over an algebraically closed field of characteristic $3$, Fermat's quartic and Klein's quartic are isomorphic, with a plane model given by $y^3-y = x^4$. Indeed, following \cite{elkies} and \cite{kanameko}, this curve is both non-classical (with $\{1,2,4\}$ as gap sequence at a generic point) and has automorphism group  isomorphic to $\mbox{PGU}(3,3)$ of order $6048$.  We refer  to \cite{VM} for several interesting results on the subject.

\s

Despite the above, if the characteristic of the ground field is big enough in comparison with the genus then the curve turns out to be classical and the Hurwitz bound holds true. Summarising, the following result is obtained.

\begin{corollary*} For every prime number $p \geqslant 5$, 
there are infinitely many pairwise non-isomorphic non-hyperelliptic projective, non-singular, absolutely irreducible algebraic curves defined over an algebraically closed field of characteristic  $p$ with automorphism group acting transitively on the \we points.
\end{corollary*}

\begin{proof} Let $X$ be a projective, non-singular, absolutely irreducible algebraic curve of genus $g$ defined over an algebraically closed field of characteristic $p$. Following \cite{sch}, if $p > 2g-2$ then $X$ is classical. In addition, as proved in  \cite{Roq}, if $X$ is non-hyperelliptic and $p > g+1$ then the Hurwitz bound  holds for $X$. In particular, if $g=3$ we can take $p \geqslant 5$ and the corollary follows from the main theorem.
\end{proof}

\subsection*{Acknowledgement} The authors are grateful to the referees for their valuable comments and suggestions.

\end{document}